\documentclass[preprint,12pt]{elsarticle}
\usepackage{latexsym}
\usepackage{amssymb}
\usepackage{amsmath}
\usepackage{amsthm}
\usepackage{verbatim}
 \usepackage[pagewise]{lineno}
\numberwithin{equation}{section}
\newtheorem{theorem}{Theorem}[section]
\newtheorem{lemma}{Lemma}[section]

\newtheorem{remark}{Remark}[section]

\allowdisplaybreaks%%%%%%%%%允许大公式断行

\biboptions{sort&compress}

\begin{document}

\begin{frontmatter}

%% Title, authors and addresses

%% use the tnoteref command within \title for footnotes;
%% use the tnotetext command for the associated footnote;
%% use the fnref command within \author or \address for footnotes;
%% use the fntext command for the associated footnote;
%% use the corref command within \author for corresponding author footnotes;
%% use the cortext command for the associated footnote;
%% use the ead command for the email address,
%% and the form \ead[url] for the home page:
%%
%% \title{Title\tnoteref{label1}}
%% \tnotetext[label1]{}
%% \author{Name\corref{cor1}\fnref{label2}}
%% \ead{email address}
%% \ead[url]{home page}
%% \fntext[label2]{}
%% \cortext[cor1]{}
%% \address{Address\fnref{label3}}
%% \fntext[label3]{}

\title{Blowup analysis of a Camassa-Holm type equation with time-varying dissipation
%\tnoteref{ack}
}
%\tnotetext[ack]{This work is partially supported by NSFC Grants (nos. 11322105 and 11671071).}

%% use optional labels to link authors explicitly to addresses:
%% \author[label1,label2]{<author name>}
%% \address[label1]{<address>}
%% \address[label2]{<address>}
\author[ad1]{Yonghui Zhou}
\ead{zhouyhmath@163.com}
\author[ad2]{Xiaowan Li}
\ead{xiaowan0207@163.com}
\author[ad3]{Shuguan Ji\corref{cor}}
\ead{jisg100@nenu.edu.cn}
\address[ad1]{School of Mathematics, Hexi University, Zhangye 734000, P.R. China}
\address[ad2]{College of Mathematics and System Sciences, Xinjiang University, Urumqi, 830046, P.R. China}
\address[ad3]{School of Mathematics and Statistics and Center for Mathematics and Interdisciplinary Sciences, Northeast Normal University, Changchun 130024, P.R. China}

\cortext[cor]{Corresponding author.}

\begin{abstract}
%% Text of abstract
This paper is concerned with the local well-posedness, wave breaking, blow-up rate for a Camassa-Holm type equation with time-dependent weak dissipation. Firstly, we obtain the local well-posedness of solutions by using Kato's theory. Secondly, by using energy estimates, characteristic methods, and comparison principles, we derive two blowup criteria involving both pointwise gradient conditions and mixed amplitude-gradient conditions, and prove the blowup rate is universally $-2$. Our results extend wave breaking analysis to physically relevant variable dissipation regimes.

\end{abstract}


\end{frontmatter}

%%
%% Start line numbering here if you want
%%
% \linenumbers

%% main text
\section{Introduction}
An intriguing phenomenon in oceanography is the occurrence of waves whose wavelengths significantly exceed the water depth. To model the propagation of such long waves in shallow water, various mathematical frameworks have been developed. Among the most prominent are the Korteweg-de Vries (KdV) equation and the Camassa--Holm (CH) equation \cite{Camassa1993}
\begin{equation}
    m_t + 2m u_x + m_x u = 0,
\label{101}
\end{equation}
where \( u(t, x) \) represents the free surface elevation. The CH equation is a nonlinear dispersive wave equation describing unidirectional irrotational shallow water waves over a flat bed.

The CH equation possesses several remarkable properties. First, it admits peakon solutions--solitary waves with a peaked crest--which have been proven to be stable \cite{Constantin2000}. Second, it exhibits wave breaking phenomena \cite{Constantin1998}; that is, smooth solutions can develop singularities in finite time in the form of wave breaking, where the solution remains bounded while its spatial derivative becomes unbounded. Furthermore, the CH equation is an infinite-dimensional integrable Hamiltonian system \cite{Fokas1981} and can be interpreted as a re-expression of geodesic flow on the diffeomorphism group \cite{Constantin2000}.

Over the past three decades, the CH equation and its generalizations have been extensively studied regarding local well-posedness, global existence of strong solutions, persistence properties, wave breaking mechanisms, peakon dynamics, and global conservative weak solutions (see, e.g., \cite{Constantin19981,Constantin1998,Constantin2000,Bressan2007,Brandolese2012,Brandolese20141,Brandolese20142,
Constantin20004,Constantin2009,Mustafa20072,Freire1,Zhou2025,Qu2020,Ji2021,Ji2022,Novruzov2022,Coclite20052,Zhou2022,Zhou2024,Zhou2026} and references therein).

In realistic coastal and estuarine environments, however, dissipative effects--such as bottom friction, viscous drag, or permeable boundaries---often play a significant role and cannot be neglected. To incorporate such energy loss, dissipative mechanisms have been introduced into these models. Early studies primarily focused on constant-coefficient dissipation for simplicity. For instance, Ott and Sudan \cite{Ott1970} examined the modification of the KdV equation due to dissipation and its effect on solitary waves. Ghidaglia \cite{Ghidaglia1988} investigated the long-time behavior of solutions to a weakly dissipative KdV equation as a finite-dimensional dynamical system. In 2009, Wu and Yin \cite{Wu2009} analyzed global existence and wave breaking for the weakly dissipative Camassa--Holm equation
\begin{equation}
    u_t - u_{txx} + 3u u_x + \lambda (u - u_{xx}) = u u_{xxx} + 2 u_x u_{xx},
    \label{102}
\end{equation}
where \(\lambda \geq 0\) is a dissipative parameter. Freire et al. \cite{Freire1} studied local well-posedness and wave breaking for a weakly dissipative Camassa-Holm type equation with quadratic and cubic nonlinearities.

Recently, growing attention has been directed toward models with time-dependent dissipative coefficients, as they more accurately capture the effects of dynamically varying environmental factors-such as tidal oscillations, fluctuating wind stress, and seasonal hydrological changes. Such models not only enhance the physical realism of shallow water wave descriptions but also provide a more flexible mathematical framework for analyzing how variable external forcing influences wave stability, propagation, and breaking phenomena.
Inspired by these advancements, this work investigates the following shallow water wave equation incorporating a time-varying dissipation term
\begin{equation}
    u_t - u_{txx} + 3u^2 u_x + \lambda(t)(u - u_{xx}) = u u_{xxx} + 2 u_x u_{xx},
    \label{103}
\end{equation}
where \(\lambda(t)\) is a continuous function.

The rest of this paper is organized as follows. In Section \ref{sec:2}, we obtain the local well-posedness of solutions by using Kato's theorem. In Section \ref{sec:3}, we derive a wave breaking mechanism of solutions and give two sufficient conditions on the initial datum for the occurrence of wave breaking.

\textbf{Notation.} Throughout this paper, all spaces of functions are over $\mathbb{R}$ and for simplicity, we drop $\mathbb{R}$ in our notation of function spaces if there is no ambiguity. Additionally, $[A,B]=AB-BA$ denotes the commutator between two operators $A$ and $B$, and $\|\cdot\|_{s}$ denotes the norm in the Sobolev space $H^{s}(\mathbb{R})$.

\section{Local well-posedness}

\setcounter{equation}{0}

\label{sec:2}

In this section, we obtain the local well-posedness by using Kato's theorem to problem \eqref{103}.
For convenience, we state Kato's theorem in the form suitable for our purpose.
Consider the abstract quasilinear evolution equation of the form
\begin{equation}
\frac{dz}{dt}+A(z)z=f(z),\ t>0,
\label{201}
\end{equation}
with the initial data $ z(0)=z_{0}.$

Let $X$ and $Y$ be Hilbert spaces such that $Y$ is continuously and densely embedded in $X$ and let $Q:Y\rightarrow X$ be a topological isomorphism. $\|\cdot\|_{X}$ and $\|\cdot\|_{Y}$ denote the norm of Banach space $X$ and $Y$, respectively. Let $L(Y, X)$ denote the space of all bounded linear operators from $Y$ to $X$ (if $Y=X$, it is abbreviated as $L(X))$.
Let $G(x)$ be the set of all negative generators of $C_0$-semigroups on $X$. More precisely, we denote by $G(X,M,\beta)$ the set of all linear operators $A$ in $X$ such that $-A$ generates a $C_0$-semigroup $\{e^{-tA}\}$ with $\|e^{-tA}\|\leq Me^{\beta t}$, $0\leq t<\infty$. $A$ is quasi-m-accretive if $A\in G(X,1,\beta)$.
Assume that:
\begin{description}
\item[$(i)$] $A(y)\in L(Y, X)$ for $y\in Y$ satisfies
$$\|(A(y)-A(z))w\|_{X}\leq \rho_{1}\|y-z\|_{X}\|w\|_{Y},\ w,y,z\in Y,$$
and $A(y)\in G(X,1,\beta)$(i.e., $A(y)$ is quasi-m-accretive), uniformly on bounded sets in $Y$;
\item[$(ii)$] $QA(y)Q^{-1}=A(y)+B(y)$, where $B(y)\in L(X)$ is bounded, uniformly on bounded sets in $Y$. Moreover,
$$\|(B(y)-B(z))w\|_{X}\leq \rho_{2}\|y-z\|_{Y}\|w\|_{X},\ y,z\in Y, w\in X;$$
\item[$(iii)$] $f:Y\rightarrow Y$ is bounded on any bounded sets in $Y$ and also extend to a map from $X$ onto $X$ and satisfies
$$\|f(y)-f(z)\|_{Y}\leq \rho_{3}\|y-z\|_{Y},\ y,z\in Y,$$
$$\|f(y)-f(z)\|_{X}\leq \rho_{4}\|y-z\|_{X},\ y,z\in X,$$
\end{description}
where $\rho_{i}(i=1,2,3)$ depend only on $\max\{\|y\|_{Y},\|z\|_{Y}\}$ and $\rho_{4}$ depends only on $\max\{\|y\|_{X},\|z\|_{X}\}.$

\begin{theorem}[Kato's theorem\cite{Kato}]
\label{the201}  Assume that (i)--(iii) hold. Given $z_{0}\in Y$, there exists a maximal $T>0$ depending only on $\|z_{0}\|_{Y}$ and unique solution $z$ to equation \eqref{201} such that
$$z=z(\cdot,z_{0})\in C([0,T);Y)\cap C^{1}([0,T);X).$$
Moreover, the map $z_{0}\mapsto z(\cdot,z_{0})$ is continuous from $Y$ to $C([0,T);Y)\cap C^{1}([0,T);X).$
\end{theorem}

In order to apply Kato's theorem, we rewrite problem \eqref{103} as follows
\begin{equation}
y_{t}+uy_{x}+2u_{x}y+\lambda(t) y+\partial_{x}h(u)=0,\ \ t>0,\ \ x\in \mathbb{R},
\label{202}
\end{equation}
where $$h(u)=u^{3}-\frac{3}{2}u^{2},\ \ \ \ y=u-u_{xx}.$$
Denote $P(x):=\frac{1}{2}e^{-|x|}$, then $(1-\partial_{x}^{2})^{-1}f=\Lambda^{-2}f=P\ast f$ for all $f\in L^{2}(\mathbb{R}), P\ast y=u,$ where $\ast$ denotes the spatial convolution operator. Using this identity, we can rewrite \eqref{202} as
\begin{equation}
u_{t}+uu_{x}+\partial_{x}\Lambda^{-2}\left(u^{2}+\frac{1}{2}u_{x}^{2}+h(u)\right)+\lambda(t) u=0,\ \ t>0,\ \ x\in \mathbb{R},
\label{203}
\end{equation}
or its equivalent form
\begin{equation}
u_{t}+uu_{x}+\partial_{x}P\ast\left(u^{2}+\frac{1}{2}u_{x}^{2}+h(u)\right)+\lambda(t) u=0,\ \ t>0,\ \ x\in \mathbb{R},
\label{204}
\end{equation}

In what follows, we prove the local well-posedness result.

\begin{theorem}[Local well-posedness]
\label{the202}  Given $u_{0}(x)\in H^{s}(s>\frac{3}{2}$), then there exists a maximal existence time of $T>0$ and a unique solution $u(t,x)$ to problem \eqref{103} such that
$$u=u(\cdot,u_{0})\in C([0,T);H^{s})\cap C^{1}([0,T);H^{s-1}).$$
Moreover, the solution $u(t,x)$ depends continuously on the initial datum, i.e., the map
$$u_{0}\mapsto u(\cdot,u_{0}):H^{s}\rightarrow C([0,T);H^{s})\cap C^{1}([0,T);H^{s-1})$$
is continuous.

% and $T$ may be chosen independent of $r$ in the following sense: if
%$$u=u(\cdot,u_{0})\in C([0,T);H^{s} )\cap C^{1}([0,T);H^{s-1})$$
%is a solution to problem \eqref{103}, and if $u_{0}(x)\in H^{s'}$ for some $s\neq s', s'>\frac{3}{2}$, then
%$$u\in C([0,T);H^{s'})\cap C^{1}([0,T);H^{s'-1})$$
%with the same $T$.
\end{theorem}
\begin{proof}
It is enough to verify that
\begin{equation*}
A(u):=u\partial_{x},\ \ f(u):=-\partial_{x}\Lambda^{-2}\left(u^{2}+\frac{1}{2}u_{x}^{2}+h(u)\right)-\lambda(t) u
\end{equation*}
satisfy conditions $(i)-(iii)$ in Kato's theorem. From the results proved in \cite{Freire1}, we can obtain that $A(u)$ satisfies conditions $(i)-(ii)$ in Kato's theorem.

Let $f(u):=g(u)-\lambda(t) u$ with $g(u)=-\partial_{x}\Lambda^{-2}\left(u^{2}+\frac{1}{2}u_{x}^{2}+h(u)\right)$. Note that for any $T'>0, \lambda(t)$ is bounded on $[0, T']$. Then for any norm $\|\cdot\|$, we have
$$\|f(u)-f(v)\|\leq\|g(u)-g(v)\|+\lambda(t)\|u-v\|.$$
This means that $f(u)$ satisfies conditions $(iii)$ in Kato's theorem if and only if $g(u)$ does, which follows from the results proved in \cite{Freire1}.
\end{proof}
\begin{theorem}
\label{the202a}  The $T$ in Theorem \ref{the201} is independent of the regularity index $s$ in the following sense. If
$$u=u(\cdot, u_{0})\in C([0,T); H^{s})\cap C^{1}([0,T); H^{s-1})$$
is a solution of equation \eqref{101} and $u_{0}(x)\in H^{s'}$ for some $s'\neq s, s'>3/2$, then
$$u\in C([0,T); H^{s'})\cap C^{1}([0,T); H^{s'-1})$$
for the same $T$. In particular, if $u_{0}(x)\in H^{\infty}=\cap_{s\geq 0}H^{s},$ then $u\in C([0,T);H^{\infty}).$
\end{theorem}

The proof of Theorem \ref{the202a} is similar to that of Theorem 2.12 in \cite{Ji2021}, so
they are omitted for simplicity.

\section{Wave breaking}
\setcounter{equation}{0}

\label{sec:3}

In this section, we derive the wave breaking mechanism of solutions for equation \eqref{103}. To this end, we firstly establish the time-dependent conserved quantity of solutions to problem \eqref{103}.

\begin{lemma}
\label{lem201}  Assume that $u(t,x)$ is a solution to problem \eqref{103} with the initial data $u_{0}(x)\in H^{s}\ (s>\frac{3}{2})$.
Let $$E(t):=\int_{\mathbb{R}}(u^{2}+u_{x}^{2})dx.$$
Then for any $t\in[0,T)$, we have
$E(t)=e^{-2\int_{0}^{t}\lambda(\tau)d\tau} E(0)$
\end{lemma}
\begin{proof}
Multiplying both sides of \eqref{103} by $2u$, we have
\begin{equation}
2uu_{t}-2uu_{txx}+6u^{3}u_{x}+2\lambda(t)(u^{2}-uu_{xx})=2u^{2}u_{xxx}+4uu_{x}u_{xx}.
\label{205}
\end{equation}
Integrating \eqref{205} with respect to $x$ over $\mathbb{R}$, we get
\begin{equation}
\frac{d}{dt}\int_{\mathbb{R}}(u^{2}+u_{x}^{2})dx+2\lambda(t) \int_{\mathbb{R}}(u^{2}+u_{x}^{2})dx=0,
\label{206}
\end{equation}
where we used the relations
\begin{equation*}
\int_{\mathbb{R}}4uu_{x}u_{xx}dx=-2\int_{\mathbb{R}}u^{2}u_{xxx}dx
\ \ \text{and}\ \
\int_{\mathbb{R}}2 uu_{xx}dx=-2\int_{\mathbb{R}}u_{x}^{2}dx.
\end{equation*}
Integrating \eqref{206} with respect to $t$ over $(0,t)$, we obtain
\begin{equation}
\int_{\mathbb{R}}(u^{2}+u_{x}^{2})dx=e^{-2\int_{0}^{t}\lambda(\tau)d\tau} \int_{\mathbb{R}}(u_{0}^{2}+u_{0x}^{2})dx,
\label{207}
\end{equation}
i.e.,
$E(t)=e^{-2\int_{0}^{t}\lambda(\tau)d\tau} E(0)$.
This completes the proof of Lemma \ref{lem201}.
\end{proof}

In what follows, we present some lemmas which is crucial in the proof of wave breaking mechanism.

\begin{lemma}[\cite{Kato1988}]
\label{lem202}
If  $r>0$, then  $H^{r}(\mathbb{R})\cap L^{\infty}(\mathbb{R})$ is an algebra. Moreover
$$\|fg\|_{r}\leq c\left(\|f\|_{L^{\infty}(\mathbb{R})}\|g\|_{r}+\|f\|_{r}\|g\|_{L^{\infty}(\mathbb{R})}\right),$$
where $c$ is a constant depending only on $r$.
\end{lemma}

\begin{lemma}[\cite{Kato1988}]
\label{lem203}
If $r>0$, then
$$\left\|\left[\Lambda^{r}, f\right] g\right\|_{L^{2}(\mathbb{R})} \leq c\left(\left\|\partial_{x} f\right\|_{L^{\infty}(\mathbb{R})}\left\|\Lambda^{r-1} g\right\|_{L^{2}(\mathbb{R})}+\left\|\Lambda^{r} f\right\|_{L^{2}(\mathbb{R})}\|g\|_{L^{\infty}(\mathbb{R})}\right),$$
where $c$ is a constant depending only on $r$.
\end{lemma}

\begin{lemma}[\cite{Constantin2002}]
\label{lem204}
If $F\in C^{\infty}(\mathbb{R})$ with $F(0)=0$. Then for any $r>1/2$, we have
$$\|F(u)\|_{r}\leq \widetilde{F}(\|u\|_{L^{\infty}})\|u\|_{r},\ u\in H^{r}(\mathbb{R}),$$
where $\widetilde{F}$ is a monotone increasing function depending only on $F$ and $r$.
\end{lemma}

\begin{theorem}
\label{the203} Let $u_{0}(x)\in H^{s} (s>3/2)$ be given and
$T$ be the maximal existence time of the solution $u(t,x)$ of equation \eqref{103} with the initial datum $u_{0}(x)$.
If there exists a positive constant $M>0$ such that
\begin{equation*}
\limsup_{t\rightarrow T}\sup_{x\in\mathbb{R}}\|{u_{x}(t,x)}\|_{L^{\infty}}\leq M,
\end{equation*}
then the $H^{s}$-norm of $u(t,\cdot)$ does not blow up on $[0,T)$.
\end{theorem}
\begin{proof}
Let $u(t,x)$ be the unique solution of equation \eqref{103} with the initial datum $u_{0}(x)\in H^{s} (s>3/2)$, which is guaranteed by Theorem \ref{the201}.

Applying the operator $\Lambda^{s}$ to equation \eqref{203}, multiplying by $2\Lambda^{s} u$ and integrating the resulting equation by parts on $\mathbb{R}$, we have
\begin{equation}
\frac{d}{dt}(u,u)_{s}=-2(uu_{x},u)_{s}-2(f(u),u)_{s}-2\lambda(t)(u,u)_{s},
\label{208}
\end{equation}
where $f(u)=\partial_{x}\Lambda^{-2}\left(u^{2}+\frac{1}{2}u_{x}^{2}+h(u)\right).$

In what follows, we estimate the right side of equality \eqref{208}. Assume there exists a positive constant $M>0$ such that
$$
\limsup_{t\rightarrow T}\sup_{x\in\mathbb{R}}\|{u_{x}(t,x)}\|_{L^{\infty}}\leq M,
$$
then we have
\begin{align}
|(uu_{x},u)_{s}|
=&|(\Lambda^{s}(uu_{x}),\Lambda^{s}u)_{L^{2}}|\nonumber\\
=&\left|\left([\Lambda^{s},u]u_{x},\Lambda^{s}u\right)_{L^{2}}
+\left(u\Lambda^{s}u_{x},\Lambda^{s}u\right)_{L^{2}}\right|\nonumber\\
\leq&\left\|[\Lambda^{s},u]u_{x}\right\|_{L^{2}}\|\Lambda^{s}u\|_{L^{2}}
+\frac{1}{2}\left\|u_{x}\right\|_{L^{\infty}}\|\Lambda^{s}u\|^{2}_{L^{2}}\nonumber\\
\leq& c_{1}\left(\|u_{x}\|_{L^{\infty}}\|\Lambda^{s-1}u_{x}\|_{L^{2}}
+\|\Lambda^{s}u\|_{L^{2}}\|u_{x}\|_{L^{\infty}}\right)\|u\|_{s}+\frac{1}{2}\left\|u_{x}\right\|_{L^{\infty}}\|u\|_{s}^{2}\nonumber\\
\leq& c_{1}\left(\|u_{x}\|_{L^{\infty}}\|u\|_{s}
+\|u_{x}\|_{L^{\infty}}\|u\|_{s}\right)\|u\|_{s}+\frac{1}{2}M\|u\|_{s}^{2}\nonumber\\
\leq&\left(2c_{1}+\frac{1}{2}\right)M\|u\|_{s}^{2},
\label{209}
\end{align}
where we used Lemma \ref{lem203} with $r=s$.
Similarly, note that $H^{s}(s>1/2)$ is a Banach algebra, it follows that
\begin{align}
&(f(u),u)_{s}\nonumber\\
=&\left(\partial_{x}\Lambda^{-2}\left(u^{2}+\frac{1}{2}u_{x}^{2}+h(u)\right),u\right)_{s}\nonumber\\
\leq& \|u\|_{s}\left\|\partial_{x}\Lambda^{-2}\left(u^{2}+\frac{1}{2}u_{x}^{2}+h(u)\right)\right\|_{s}
\nonumber\\
\leq& \|u\|_{s}\left(\left\|u^{2}\right\|_{s-1}+\left\|u_{x}^{2}\right\|_{s-1}+\left\|h(u)\right\|_{s-1}
\right)\nonumber\\
\leq& \|u\|_{s}\bigg(2c_{2}\|u\|_{L^{\infty}}\|u\|_{s-1}+2c_{3}\|u_{x}\|_{L^{\infty}}\|u_{x}\|_{s-1}
+\widetilde{h}(\|u\|_{L^{\infty}})\left\|u\right\|_{s-1}
\bigg)\nonumber\\
\leq&c_{4}\|u\|_{s}^{2},
\label{2010}
\end{align}
where $c_{4}=2c_{2}\sqrt{E(0)}+2c_{3}M+\widetilde{h}(\sqrt{E(0)}).$
Here we used Lemmas \ref{lem202} and \ref{lem204} with $r=s$.

Substituting \eqref{209} and \eqref{2010} into \eqref{208}, we get
\begin{equation}
\frac{d}{dt}\|u\|_{s}^{2}\leq \left(c_{5}-2\lambda(t)\right)\|u\|_{s}^{2},
\label{2011}
\end{equation}
where $c_{5}=\left(4c_{1}+1\right)M+2c_{4}.$
By using Gronwall's inequality, we have
$$\|u\|_{s}^{2}\leq e^{c_{5}t-2\int_{0}^{t}\lambda(\tau)d\tau}\|u_{0}\|_{s}^{2}.$$
This completes the proof of Theorem \ref{the203}.
\end{proof}

%Below, we will establish a relation between the behavior of $u_{x}$ and wave breaking phenomena of solutions to problem \eqref{103}, and then give a sufficient condition on the initial data to guarantee the occurrence of wave breaking.
%In the following proof, we only prove the case of $r=3$, since we can obtain the same conclusion for general case $r>\frac{3}{2}$ by using denseness.
\begin{theorem}
\label{the204} Let $u_{0}(x)\in H^{s}(s>\frac{3}{2})$ be given, and
$T$ be the maximal existence time of the corresponding solution $u(t,x)$ to problem \eqref{103}.
Then $T$ is finite if and only if
\begin{equation*}
\liminf_{t\rightarrow T}\inf_{x\in\mathbb{R}}{u_{x}(t,x)}=-\infty.
\end{equation*}
\end{theorem}
\begin{proof}
Since the maximal existence time $T$ is independent of the choice of $s$ in view of Theorem \ref{the202a}, we only need to consider the case $s=2$ by using a simple density argument.

Note that $y=u-u_{xx}$, then we have
\begin{equation*}
\|y\|_{L^{2}}^{2}=\int_{\mathbb{R}}(u-u_{xx})^{2}dx=\int_{\mathbb{R}}(u^{2}+u_{xx}^{2}+2u_{x}^{2})dx,
\end{equation*}
%and
%\begin{equation*}
%\|y_{x}\|_{L^{2}}^{2}=\int_{\mathbb{S}}(u_{x}-u_{xxx})^{2}dx
%=\int_{\mathbb{S}}(u_{x}^{2}+u_{xxx}^{2}+2u_{xx}^{2})dx,
%\end{equation*}
where we used the fact that
$
\int_{\mathbb{R}}2uu_{xx}dx
=-2\int_{\mathbb{R}}u_{x}^{2}dx.
%\ \ \text{and}\ \
%\int_{\mathbb{S}}2u_{x}u_{xxx}dx
%=-2\int_{\mathbb{S}}u_{xx}^{2}dx.
$
Therefore, we can conclude that
\begin{equation}
\|u\|_{2}^{2}\leq\|y\|_{L^{2}}^{2}\leq 2\|u\|_{2}^{2}.
\label{2012}
\end{equation}

Multiplying $2y$ to both sides of equation \eqref{202}, we get
\begin{equation}
2yy_{t}+4y^{2} u_{x}+2yy_{x}u+2y\partial_{x}h(u)+2\lambda(t) y^{2}=0.
\label{2013}
\end{equation}
Integrating \eqref{2013} with respect to\ $x$ over $\mathbb{R}$, we obtain
\begin{align}
\frac{d}{dt}\int_{\mathbb{R}}y^{2}dx&=-\int_{\mathbb{R}}(4y^{2} u_{x}+2yy_{x}u+2y\partial_{x}h(u)+2\lambda(t) y^{2})dx\nonumber\\
&=-\int_{\mathbb{R}}(3y^{2} u_{x}+2y\partial_{x}h(u)+2\lambda(t) y^{2})dx,
\label{2014}
\end{align}
where we used the relation
$
2\int_{\mathbb{R}}yy_{x}udx=-\int_{\mathbb{R}}y^{2}u_{x}dx.
$

Note that
$$2\int_{\mathbb{R}}y\partial_{x}h(u)dx\leq \int_{\mathbb{R}}(y^{2}+(\partial_{x}h(u))^{2})dx$$
and
$$\int_{\mathbb{R}}(\partial_{x}h(u))^{2}dx\leq \|h(u)\|_{1}^{2}\leq c_{4}\|u\|_{1}^{2}\leq c_{4}\|y\|_{L^{2}}^{2},$$
where $c_{4}>0$ represents a specific constant.

If $u_{x}(t,x)$ is bounded from below on $[0,T)\times \mathbb{R}$, i.e., there exists $N>0$ such that $u_{x}(t,x)>-N$ on $[0,T)\times \mathbb{R}$. Then we get
\begin{equation*}
\frac{d}{dt}\int_{\mathbb{R}}y^{2}dx\leq\left(3N+1-2\lambda(t)+c_{4}\right)\|y\|_{L^{2}}^{2},
\end{equation*}
By using Gronwall's inequality, we get
\begin{equation*}
\|y\|_{L^{2}}^{2}\leq e^{(3N+1+c_{4})t-2\int_{0}^{t}\lambda(\tau)d\tau}\|y_{0}\|_{L^{2}}^{2}.
\end{equation*}
Thus, by the Sobolev embedding theorem, it follows that
\begin{equation*}
\|u_{x}\|_{L^{\infty}}\leq \|u\|_{2}\leq \|y\|_{L^{2}}^{2}\leq e^{(3N+1+c_{4})t-2\int_{0}^{t}\lambda(\tau)d\tau}\|y_{0}\|_{L^{2}}^{2}.
\end{equation*}
By applying Theorem \ref{the203}, we deduce that the solution exists globally in time.

On the other hand, if $u_{x}(t,x)$ is unbounded from below, by applying Theorem \ref{the203} and using the Sobolev embedding theorem, we infer that the solution will blow up in finite time.

This completes the proof of Theorem \ref{the204}.
\end{proof}

In what follows, we give the first blowup criterion. To this end, we need the following lemma.
\begin{lemma}[\cite{Constantin1998}]
\label{lem205} Let $T>0$ and $v(t,x)\in C^{1}([0,T);H^{2}(\mathbb{R}))$ be a given function. Then, for any $t\in[0,T)$, there exists at least one point $\xi(t)\in\mathbb{R}$ such that
\begin{equation*}
m(t)=\inf_{x\in \mathbb{R}}v_{x}(t,x)=v_{x}(t,\xi(t)),
\end{equation*}
and the function $m(t)$ is almost everywhere differentiable in $[0,T)$, with
\begin{equation*}
m'(t)=v_{tx}(t,\xi(t)),\ \ a.e.\ on\ [0,T).
\end{equation*}
\end{lemma}
\begin{theorem}
\label{the205} Let $u_{0}\in H^{s} (s>\frac{3}{2})$ be given and $T$
be the maximal existence time of the corresponding solution $u(t,x)$ to problem \eqref{103}.
Assume that $\lambda(t)\leq \delta$ and some $x_{0}\in \mathbb{R}$ such that
\begin{equation}
u_{0}'(x_{0})<-\delta-\sqrt{\delta^{2}+2K},
\label{2015}
\end{equation}
where $K=\frac{\sqrt{2}}{2}\|u_{0}\|^{3}_{1}+\frac{5}{2}\|u_{0}\|^{2}_{1}.$ Then the corresponding solution of the Cauchy problem \eqref{103} blows up in finite time.
\end{theorem}
\begin{proof}
Differentiating  \eqref{203} with respect to $x$, we get
\begin{equation}
u_{tx}+u_{x}^{2}+uu_{xx}+\partial_{x}^{2}\Lambda^{-2}(u^{2}+\frac{u_{x}^{2}}{2}+h(u))+\lambda(t) u_{x}=0.
\label{2016}
\end{equation}
According the relation $\partial_{x}^{2}\Lambda^{-2}=\Lambda^{-2}-1$, we obtain
\begin{equation}
u_{tx}+\frac{1}{2}u_{x}^{2}+uu_{xx}+P\ast\left(u^{2}+\frac{u_{x}^{2}}{2}\right)-u^{2}
+P\ast h(u)-h(u)+\lambda(t) u_{x}=0.
\label{2017}
\end{equation}
According to Lemma \ref{lem205} and Theorem \ref{the201}, there is at least one point $\xi(t)\in \mathbb{R}$ satisfying
$$u_{x}(t,\xi(t))=\inf_{x\in\mathbb{R}}u_{x}(t,x).$$
Let
\begin{equation*}
m(t)=\inf_{x\in \mathbb{R}}u_{x}(t,x)=u_{x}(t,\xi(t)).
\end{equation*}
Then for any given $t\in(0,T)$, we obtain $u_{xx}(t,\xi(t))=0$. Therefore, we have
\begin{equation}
m'(t)=-\frac{1}{2}m^{2}-\lambda(t) m +u^{2}-P\ast\left(u^{2}+\frac{u_{x}^{2}}{2}\right)
-P\ast h(u)+h(u),\ \ a.e.\ on\ (0,T).
\label{2018}
\end{equation}
Note that
\begin{align*}
\left\|P\ast\left(u^{2}+\frac{u_{x}^{2}}{2}\right)\right\|_{L^{\infty}}\leq
\|P\|_{L^{\infty}}\left\|u^{2}+\frac{u_{x}^{2}}{2}\right\|_{L^{1}}
\leq\frac{1}{2}\|u\|_{1}^{2}
\leq\frac{1}{2}\|u_{0}\|_{1}^{2},
\end{align*}
\begin{align*}
u^{2}(t,\xi(t))&=\Big|\int_{-\infty}^{\xi(t)}u(t,y)u_{x}(t,y)dy
-\int_{\xi(t)}^{+\infty}u(t,y)u_{x}(t,y)dy\Big|\nonumber\\
&\leq\frac{1}{2}\int_{-\infty}^{\xi(t)}(u^{2}(t,y)+u_{x}^{2}(t,y))dy
+\frac{1}{2}\int_{\xi(t)}^{+\infty}(u^{2}(t,y)+u_{x}^{2}(t,y))dy\nonumber\\
&=\frac{1}{2}\|u\|_{1}^{2}\leq\frac{1}{2}\|u_{0}\|_{1}^{2},
\end{align*}
thus we get $|u(t,x)|\leq \frac{\sqrt{2}}{2}\|u_{0}\|_{1}.$ Then we have
\begin{equation*}
|h(u)|\leq \frac{\sqrt{2}}{4}\|u_{0}\|^{3}_{1}+\frac{3}{4}\|u_{0}\|^{2}_{1}
\ \ \text{and}\ \
|P\ast h(u)|\leq \frac{\sqrt{2}}{4}\|u_{0}\|^{3}_{1}+\frac{3}{4}\|u_{0}\|^{2}_{1}.
\end{equation*}
Therefore, \eqref{2018} is reduced to
\begin{align}
m'(t)&\leq-\lambda(t) m(t)-\frac{1}{2}m^{2}(t)+K\nonumber\\
&=-\frac{1}{2}(m+\lambda(t)+\sqrt{\lambda^{2}(t)+2K})(m+\lambda(t)-\sqrt{\lambda^{2}(t)+2K}),\ a.e.\ on\ (0,T),
\label{2019}
\end{align}
where $$K=\frac{\sqrt{2}}{2}\|u_{0}\|^{3}_{1}+\frac{5}{2}\|u_{0}\|^{2}_{1}.$$
By \eqref{2015}, we have
$$m(0)<-\delta-\sqrt{\delta^{2}+2K},$$
thereby $m'(0)<0$. From the continuity of $m(t)$ with respect to $t$, it can be obtained that for any $t\in[0,T)$, there is $m'(t)<0$. Therefore
$$m(t)<-\delta-\sqrt{\delta^{2}+2K}.$$
For $m(t)<0$, we have
\begin{eqnarray}
m'(t)\leq-\delta m(t)-\frac{1}{2}m^{2}(t)+K,\ \ a.e.\ on\ (0,T).
\label{2020}
\end{eqnarray}
Let $\omega(t)$ satisfy
\begin{align}
\omega'(t)&=-\delta \omega(t)-\frac{1}{2}\omega^{2}(t)+K\nonumber\\
&=-\frac{1}{2}(\omega+\delta+\sqrt{\delta^{2}+2K})(\omega+\delta-\sqrt{\delta^{2}+2K}),\ a.e.\ on\ (0,T),
\label{2021}
\end{align}
By the comparison principle, we have $m(t)\leq \omega(t)$. Under the given initial condition, $\omega(t)$ blows up in infinite time, therefore $m(t)$ must also blow up in finite time.

For the comparison equation \eqref{2021}, from the initial condition, we have
$$\omega(0)<-\delta-\sqrt{\delta^{2}+2K},$$
 thereby $\omega'(0)<0$. From the continuity of $\omega(t)$ with respect to $t$, we can obtain that $\omega'(t)<0$ for any $t\in[0,T)$. Therefore
$$\omega(t)<-\delta-\sqrt{\delta^{2}+2K}.$$
Then, by solving the above inequality, we obtain
\begin{equation*}
\frac{\omega(0)+\delta+\sqrt{\delta^{2}+2K}}
{\omega(0)+\delta-\sqrt{\delta^{2}+2K}}
e^{\sqrt{\delta^{2}+2K}t}-1
\leq\frac{2\sqrt{\delta^{2}+2K}}
{\omega(t)+\delta-\sqrt{\delta^{2}+2K}}\leq 0.
\end{equation*}
Due to
$$0<\frac{\omega(0)+\delta+\sqrt{\delta^{2}+2K}}
{\omega(0)+\delta-\sqrt{\delta^{2}+2K}}<1,$$
there exists
$$T_{\omega}\leq \frac{1}{\sqrt{\delta^{2}+2K}}\log \frac{\omega(0)+\delta-\sqrt{\delta^{2}+2K}}
{\omega(0)+\delta+\sqrt{\delta^{2}+2K}},$$
such that $$\liminf_{t\rightarrow T}\omega(t)=-\infty.$$
Since $m(t)\leq \omega(t)$, the blowup time $T$ for the original problem satisfies
$$T_{1}^{*}\leq T_{\omega}\leq \frac{1}{\sqrt{\delta^{2}+2K}}\log \frac{\omega(0)+\delta-\sqrt{\delta^{2}+2K}}
{\omega(0)+\delta+\sqrt{\delta^{2}+2K}}.$$
According to Theorem \ref{the203}, the corresponding solution $u(t,x)$ to problem \eqref{103} blows up in finite time. This completes the proof of Theorem \ref{the205}.
\end{proof}
\begin{theorem}
\label{the206}
Under the conditions of Theorem \ref{the205}. If $T_{1}^{\ast}$ is finite, then
\begin{equation*}
\liminf_{t\rightarrow T_{1}^{\ast}}\inf_{x\in \mathbb{R}}u_{x}(t,x)(T_{1}^{\ast}-t)=-2.
\end{equation*}
\end{theorem}
\begin{proof}
The argument of proof is divided into six steps.

\noindent\textbf{Step 1. Key differential inequality.}
From \eqref{2018}, we obtain that $m(t)$ satisfies
\begin{align}
m'(t) +\frac{1}{2} \big( m(t) + \lambda(t) \big)^2 = H(t),
\label{2022}
\end{align}
where $$H(t)=u^{2}-P\ast\left(u^{2}+\frac{u_{x}^{2}}{2}\right)
-P\ast h(u)+h(u)+\frac{1}{2}\lambda^{2}(t).$$ It is easy check that $H(t)$ is uniformly bounded: $|H(t)| \leq K+\frac{\delta^{2}}{2}:=K_{1}$ for all $t \in [0, T_{1}^*)$.

\noindent\textbf{Step 2. Introduction of auxiliary variables.}
Let $$z(t) = m(t) + \lambda(t).$$ Then $z(t) \to -\infty$ as $t \to T_{1}^*$. By the continuity of $\lambda(t)$, the limit $\bar\lambda = \lim_{t \to T_{1}^*} \lambda(t)$ exists. For any $\varepsilon > 0$, there exists $\delta > 0$ such that $$|\lambda(t) - \bar\lambda| < \varepsilon,\text{ for }T_{1}^* - t < \delta.$$

Define $\theta(t) = m(t) + \bar\lambda$. Then $\theta(t) \to -\infty$, and $z(t) = \theta(t) + (\lambda(t) - \bar\lambda)$.

\noindent\textbf{Step 3. Estimating the difference $z^2(t) - \theta^2(t)$.}
We compute
\begin{equation}
|z^2(t) - \theta^2(t)| = |2\theta(t)(\lambda(t) - \bar\lambda) + (\lambda(t) - \bar{\lambda})^2|
\leq 2\varepsilon |\theta(t)| + \varepsilon^2.
\label{2023}
\end{equation}

\noindent\textbf{Step 4. Reformulating the equation for $\theta(t)$.}
From \eqref{2022} and $m'(t) = \theta'(t)$, we have
$\theta'(t) = -\frac{1}{2} z^2(t) + H(t).$
Using \eqref{2023}, we obtain
\begin{align}
\theta'(t) &\leq -\frac{1}{2}\theta^2(t) + \varepsilon |\theta(t)| + \frac{1}{2} \varepsilon^2 + K_{1},\nonumber\\
\theta'(t) &\geq -\frac{1}{2}\theta^2(t) - \varepsilon |\theta(t)| -\frac{1}{2} \varepsilon^2 - K_{1}.
\label{2024}
\end{align}
Since $\theta(t) < 0$, let $\phi(t) = -\theta(t) > 0$. Then $\phi(t) \to +\infty$ and \eqref{2024} becomes
\begin{align}
\phi'(t) &\geq \frac{1}{2} \phi^2(t) - \varepsilon \phi(t) - K_2,\nonumber\\
\phi'(t) &\leq \frac{1}{2}\phi^2(t) + \varepsilon \phi(t) + K_2,
\label{2025}
\end{align}
where $K_2 = K_1 + \frac{1}{2} \varepsilon^2$.

%\noindent\textbf{Step 5. Dominant balance and bounds for $\phi'(t)$.}
%Since $\phi(t) \to +\infty$, we can choose $t_1 < T^*$ such that for all $t \in [t_1, T^*)$, the following inequality holds:
%$$\frac{1}{4} \phi^2 (t)\geq \varepsilon \phi(t) + M_1.
%$$
%This is possible because the left-hand side grows quadratically while the right-hand side grows linearly in $\theta$. Under condition (5.1), we analyze the inequalities in (5).

%From the lower bound in (5):

%\begin{equation}
%\phi'(t) \geq \frac{1}{2} \phi^2(t) - \varepsilon \phi(t) - M_1
%=\frac{1}{4} \phi^2(t) + \left( \frac{1}{4} \phi^2(t) - \varepsilon \phi(t) - M_1 \right)
%\geq \frac{1}{4}\phi(t)^2 .
%\end{equation}

%From the upper bound in (5):

%\begin{equation}
%\phi'(t) \leq \frac{1}{2} \phi^2(t) + \varepsilon \phi(t) + M_1
%=\frac{3}{4} \phi^2(t) - \left(\frac{1}{4} \phi^2(t) - \varepsilon \phi(t) - M_1 \right)
%\leq \frac{3}{4} \phi(t)^2.
%\end{equation}

%Combining (5.2) and (5.3) yields:
%\[
%\frac{1}{4} \phi^2(t) \leq \phi'(t) \leq \frac{3}{4} \phi^2(t). \tag{6}
%\]

\noindent\textbf{Step 5. Precise blow-up rate for $\phi(t)$.}
Dividing \eqref{2025} by $\phi^2(t)$ gives
\begin{align}
\frac{1}{2} - \frac{\varepsilon}{\phi(t)} - \frac{K_2}{\phi^2(t)} \leq \frac{\phi'(t)}{\phi^2(t)} \leq \frac{1}{2} + \frac{\varepsilon}{\phi(t)} + \frac{K_2}{\phi^2(t)}.
\label{2026}
\end{align}
Since $\phi(t) \to +\infty$, for any $\varepsilon' > 0$, there exists $t_2 \in [t_1, T_{1}^*)$ such that for $t \in [t_2, T_{1}^*)$,
\begin{align}
\frac{1}{2} - \varepsilon' \leq \frac{\phi'(t)}{\phi^2(t)} \leq \frac{1}{2} + \varepsilon'.
\label{2027}
\end{align}
Note that $\frac{\phi'(t)}{\phi^2(t)} = -\frac{d}{dt} \left( \frac{1}{\phi(t)} \right)$. Integrating \eqref{2027} with respect to $t$ from $t$ to $T_{1}^*$ yields
$$ \left( \frac{1}{2} - \varepsilon' \right)(T_{1}^* - t) \leq \frac{1}{\phi(t)} \leq \left( \frac{1}{2} + \varepsilon' \right)(T_{1}^* - t).
$$
Hence,
\begin{align}
\lim_{t \to T_{1}^*} \phi(t)(T_{1}^* - t) = 2.
\label{2028}
\end{align}

\noindent\textbf{Step 6. Returning to $m(t)$.}
Since $\phi(t) = -m(t) - \bar\lambda$, we have
\begin{align*}
-m(t) - \bar\lambda \sim \frac{2}{T_{1}^* - t},
\end{align*}
which implies
\begin{align*}
m(t)(T_{1}^* - t) \sim -2 - \bar\lambda (T_{1}^* - t) \to -2.
\end{align*}
Therefore,
\begin{align*}
\liminf_{t \to T_{1}^*} m(t)(T_{1}^* - t) = -2,
\end{align*}
i.e.,
\begin{align*}
\liminf_{t \to T_{1}^*} \inf_{x\in \mathbb{R}}u_{x}(t) \cdot (T_{1}^* - t) = -2.
\end{align*}
This completes the proof.
\end{proof}

In what follows, we track the dynamics of some linear combinations of $u$ and $u_{x}$ along the characteristic emanating from $x_1$. Given the solution of \eqref{204}, the characteristic of \eqref{204} is defined by the following differential equation
\begin{equation}
\begin{cases}
q_{t}(t,x)=u(t,q(t,x)),\ (t,x)\in[0,T)\times\mathbb{R},\\
q(0,x)=x,\ x\in \mathbb{R}.
\end{cases}
\label{2029}
\end{equation}
A direct calculation shows that
\begin{equation*}
\begin{cases}
q_{tx}(t,x)=u_{x}(t,q(t,x))q_{x}(t,x),\ (t,x)\in[0,T)\times\mathbb{R},\\
q_{x}(0,x)=1,\ x\in \mathbb{R}.
\end{cases}
\end{equation*}
Applying the classical results in the theory of ODEs, we can prove that the map $q(t,\cdot)$ is an increasing diffeomorphism of  $\mathbb{R}$ with
$$q_{x}(t,x)=e^{\int_{0}^{t}u_{x}(\tau,q(\tau,x))d\tau}>0,\ (t,x)\in [0,T)\times \mathbb{R}.$$

Let us denote $'$ to be the derivative $\partial_{t}+u\partial_{x}$. Then the dynamics of $u$ and $u_{x}$ along the characteristics $q(t,x)$ are formulated in the following lemma.
\begin{lemma}
\label{lem303}  Let $u_{0}(x)\in H^{s}\ (s>\frac{3}{2})$. Then $u(t,q(t,x))$ and $u_{x}(t,q(t,x))$ satisfying the following integro-differential equations
\begin{align*}
u'(t)=&P_{+}\ast\left(u^{2}+\frac{1}{2}u_{x}^{2}+h(u)\right)-P_{-}\ast\left(u^{2}+\frac{1}{2}u_{x}^{2}+h(u)\right)-\lambda(t) u,\\
u_{x}'(t)=&-\frac{1}{2}u_{x}^{2}-P_{+}\ast\left(u^{2}+\frac{1}{2}u_{x}^{2}+h(u)\right)-P_{-}\ast\left(u^{2}+\frac{1}{2}u_{x}^{2}
+h(u)\right)+u^{2}\\
&+h(u)-\lambda(t) u_{x}.
\end{align*}
\end{lemma}
\begin{proof}
Define the two convolution operators $P_{+}$ and $P_{-}$ as
$$P_{+}\ast f(x)=\frac{1}{2}e^{-x}\int_{-\infty}^{x}e^{y}f(y)dy,\ \
P_{-}\ast f(x)=\frac{1}{2}e^{x}\int^{+\infty}_{x}e^{-y}f(y)dy,$$
then $P=P_{+}+P_{-}$ and $P_{x}=P_{-}-P_{+}$.
From \eqref{204} and \eqref{2016}, we get
\begin{align*}
u'(t)=&u_{t}+uu_{x}\nonumber\\
=&-\partial_{x}P\ast\left(u^{2}+\frac{1}{2}u_{x}^{2}+h(u)\right)-\lambda(t) u\nonumber\\
=&P_{+}\ast\left(u^{2}+\frac{1}{2}u_{x}^{2}+h(u)\right)-P_{-}\ast\left(u^{2}+\frac{1}{2}u_{x}^{2}+h(u)\right)-\lambda(t) u,\nonumber\\
u_{x}'(t)=&u_{tx}+uu_{xx}\nonumber\\
=&-\frac{1}{2}u_{x}^{2}-P\ast\left(u^{2}+\frac{u_{x}^{2}}{2}+h(u)\right)+u^{2}
+h(u)-\lambda(t) u_{x}\nonumber\\
=&-\frac{1}{2}u_{x}^{2}-P_{+}\ast\left(u^{2}+\frac{1}{2}u_{x}^{2}+h(u)\right)-P_{-}\ast\left(u^{2}+\frac{1}{2}u_{x}^{2}
+h(u)\right)+u^{2}\\
&+h(u)-\lambda(t) u_{x}.
\end{align*}
This completes the proof of lemma \ref{lem303}.
\end{proof}

In what follows, we give the second blowup criteria. To this end, we need the following lemma.
\begin{lemma} \cite{Chen2016}
\label{lem207}Let $f(t)\in C^{1}(\mathbb{R}), a>0, b>0$ and $f(0)>\sqrt{\frac{b}{a}}$. If $f'(t)\geq af^{2}(t)-b$, then
$$f(t)\rightarrow +\infty,\ \ as\ \ t\rightarrow T^{*}\leq \frac{1}{2\sqrt{ab}}log\frac{f(0)+\sqrt{\frac{b}{a}}}{f(0)-\sqrt{\frac{b}{a}}}.$$
\end{lemma}

In what follows, we give the second blowup criterion for problem \eqref{202}.
\begin{theorem}
\label{the207} Let $u_{0}(x)\in H^{s}$ and $\lambda(t)\leq \delta$ for $\delta>0$. Assume that there exists a point $x_{1}\in \mathbb{R}$ such that
\begin{equation}
u_{0x}(x_{1})
<-|u_{0}(x_{1})|-\delta-\sqrt{\delta^{2}+2K},
\label{2030}
\end{equation}
where $K=\frac{\sqrt{2}}{2}\|u_{0}\|^{3}_{1}+\frac{5}{2}\|u_{0}\|^{2}_{1}$.
Then the corresponding solution $u(t,x)$ of \eqref{204} with the initial datum $u_{0}(x)$ blows up in finite time with an estimate of the blowup
time $T_{2}^{\ast}$ as
$$T_{2}^{\ast}\leq\frac{1}{\sqrt{\delta^{2}+2K}}\log\frac{g(0)-\delta+\sqrt{\delta^{2}+2K}}{g(0)-\delta-\sqrt{\delta^{2}+2K}}.$$
\end{theorem}
\begin{proof}
Let $T$ be the maximal existence time of the unique solution $u(t,x)$ guaranteed by Theorem \ref{the202}.

In order to obtain the blowup result, we define two $C^{1}-$differentiable functions $\Phi(t)$ and $\Psi(t)$ as
$$\Phi(t)=(u-u_{x})(t,q(t,x_{1})),\ \ \ \Psi(t)=(u+u_{x})(t,q(t,x_{1})).$$
Then we have
\begin{align}
\Phi'(t)=&u'(t)-u'_{x}(t)\nonumber\\
=&\frac{1}{2}u_{x}^{2}+2P_{+}\ast\left(u^{2}+\frac{1}{2}u_{x}^{2}+h(u)\right)-u^{2}-h(u)-\lambda(t) (u-u_{x}).\label{2031}
\\
\Psi'(t)=&u'(t)+u'_{x}(t)\nonumber\\
=&-\frac{1}{2}u_{x}^{2}-2P_{-}\ast\left(u^{2}+\frac{1}{2}u_{x}^{2}+h(u)\right)+u^{2}+h(u)-\lambda(t) (u+u_{x}).
\label{2032}
\end{align}
Note that
\begin{align*}
P_{+}\ast\left(u^{2}+u_{x}^{2}\right)=&\frac{1}{2}e^{-x}\int_{-\infty}^{x}e^{\xi}\left(u^{2}+u_{\xi}^{2}\right)d\xi
\geq e^{-x}\int_{-\infty}^{x}e^{\xi}uu_{\xi}d\xi
=\frac{1}{2}u^{2}-P_{+}\ast u^{2},
\end{align*}
thus
$$P_{+}\ast\left(2u^{2}+u_{x}^{2}\right)\geq \frac{1}{2}u^{2},\ \ \text{i.e.,}\ \ P_{+}\ast\left(u^{2}+\frac{1}{2}u_{x}^{2}\right)\geq \frac{1}{4}u^{2}.$$
Similarly,
 $$P_{-}\ast\left(u^{2}+\frac{1}{2}u_{x}^{2}\right)\geq \frac{1}{4}u^{2},$$
$$|h(u)|,\ \ |P_{+}\ast h(u)|\leq\frac{\sqrt{2}}{4}\|u_{0}\|_{1}^{3}+\frac{3}{4}\|u_{0}\|_{1}^{2}.$$
Therefore, \eqref{2031} and \eqref{2032} are reduced to
\begin{align}
\Phi'(t)\geq &\frac{1}{2}u_{x}^{2}-\frac{1}{2}u^{2}+2P_{+}\ast h(u)-h(u)-\lambda(t) (u-u_{x})\nonumber\\
\geq&-\frac{1}{2}\Phi \Psi-\lambda(t)\Phi-K\nonumber\\
=&-\frac{1}{2}\Phi(\Psi+2\lambda(t))-K,\label{2033}
\\
\Psi'(t)\leq&-\frac{1}{2}u_{x}^{2}+\frac{1}{2}u^{2}-2P_{-}\ast h(u)+h(u)-\lambda(t)\Psi\nonumber\\
\leq&\frac{1}{2}\Phi \Psi-\lambda(t)\Psi+K\nonumber\\
=&\frac{1}{2}\Psi(\Phi+2\lambda(t))+K.\label{2034}
\end{align}

In what follows, we consider the two equations associated to the equations \eqref{2033}--\eqref{2034} as
\begin{align}
\bar{\Phi}'(t)=-\frac{1}{2}\Phi(\Psi+2\delta)-K,\ \
\bar{\Psi}'(t)=\frac{1}{2}\Psi(\Phi+2\delta)+K,
\label{2035}
\end{align}
with $\bar{\Phi}(0)=\Phi(0),\ \bar{\Psi}(0)=\Psi(0)$.

By the comparison principle, we can conclude that if the solution of \eqref{2035} blows up in finite time, then the solution of \eqref{2033}-\eqref{2034} will also blow up in finite time.
From \eqref{2030}, we get
$$\bar{\Phi}(0)>\delta+\sqrt{\delta^{2}+2K}>0,\ \ \bar{\Phi}'(0)>0,
$$
$$\bar{\Psi}(0)<-\delta-\sqrt{\delta^{2}+2K}<0,\ \ \bar{\Psi}'(0)<0.
$$
Therefore, over the time of existence along the characteristic emanating from $x_{1}$, it always holds that
$$\bar{\Phi}(t)>0,\ \ \bar{\Psi}(t)<0,\ \ \bar{\Phi}\bar{\Psi}(t)<-(\delta+\sqrt{\delta^{2}+2K})^{2}<0.$$

In what follows, we consider the evolution of the quantity $g(t)=\sqrt{-\bar{\Phi}\bar{\Psi}(t)}$. Since $\bar{\Phi}-\bar{\Psi}\geq 2g$, then we get
\begin{align}
g'(t)=&-\frac{\bar{\Phi}'\bar{\Psi}+\bar{\Phi}\bar{\Psi}'}{2\sqrt{-\bar{\Phi}\bar{\Psi}(t)}}\nonumber\\
=&-\frac{\left(-\frac{1}{2}\bar{\Phi}\bar{\Psi}-K-\delta \bar{\Phi}\right)\bar{\Psi}
+\bar{\Phi}\left(\frac{1}{2}\bar{\Phi}\bar{\Psi}+K-\delta \bar{\Psi}\right)}{2\sqrt{-\bar{\Phi}\bar{\Psi}}}\nonumber\\
=&-\frac{1}{2\sqrt{-\bar{\Phi}\bar{\Psi}}}\left(K+\frac{1}{2}\bar{\Phi}\bar{\Psi}\right)(\bar{\Phi}-\bar{\Psi})+\frac{\delta \bar{\Phi}\bar{\Psi}}{\sqrt{-\bar{\Phi}\bar{\Psi}}}
\nonumber\\
\geq&\frac{1}{2}g^{2}-\delta g-K\nonumber\\
=&\frac{1}{2}(g-\delta)^{2}-\frac{\delta^{2}}{2}-K.
\label{2036}
\end{align}
The monotonicity of $\bar{\Phi}(t)$ and $\bar{\Psi}(t)$ implies that $g(t)=\sqrt{-\bar{\Phi}\bar{\Psi}(t)}$ is increasing with $t\in[0,T)$, thus
\begin{equation}
g(t)=\sqrt{-\bar{\Phi}\bar{\Psi}(t)}\geq\sqrt{-\bar{\Phi}\bar{\Psi}(0)}>(\delta+\sqrt{\delta^{2}+2K})^{2}>0.
\label{2037}
\end{equation}
Therefore, we get
$$\frac{1}{2}(g-\delta)^{2}-\frac{\delta^{2}}{2}-K>0.$$
From lemma \ref{lem207}, we can conclude that $\lim_{t\rightarrow T_{2}^{\ast}}g(t)=+\infty$. Moreover, the maximal existence time $T_{2}^{\ast}$ can be given by
$$T_{2}^{\ast}\leq\frac{1}{\sqrt{\delta^{2}+2K}}\log\frac{g(0)-\delta+\sqrt{\delta^{2}+2K}}{g(0)-\delta-\sqrt{\delta^{2}+2K}},$$
where $g(0)=\sqrt{u_{0x}^{2}(x_{1})-u_{0}^{2}(x_{1})}$.
Since $g(t)\leq \frac{\bar{M}-\bar{N}}{2}=-u_{x}(t,q(t,x_{1}))$, we get
$$\lim_{t\rightarrow T_{2}^{\ast}}u_{x}(t,q(t,x_{1}))=-\infty.$$

From the above argument, we can obtain the information for the location point of blowup. In fact, this information is a consequence of the elementary calculus inequality for continuously derivable functions
$$f: |f(t)- f(0)| \leq \kappa |t|,\ \kappa=\sup |f'|.$$
We can to apply this relation to the function $f(t)=q(t,x_0)$.
From \eqref{2029}, we get
$$\kappa\leq \|u(t)\|_{L^{\infty}}\leq \frac{1}{\sqrt 2}\|u\|_{1}\leq\frac{1}{\sqrt 2}\sqrt{E(0)}.$$
Moreover, $f(0)=x_1$.
So, we deduce, for $0<t<T^{*}$,
$$|q(t,x_1)-x_1|\leq\frac{1}{\sqrt 2}\sqrt{E(0)}t.$$
Hence, the blowup point $q(T^{*},x_1)$ must be inside the interval
$$\left[x_{1}-\frac{1}{\sqrt 2}\sqrt{E(0)}T^{*},
x_{1}+\frac{1}{\sqrt 2}\sqrt{E(0)}T^{*}\right].$$
This completes the proof of theorem \ref{the207}.
\end{proof}
\begin{remark}
Theorem \ref{the205} is a pure slope criterion, while Theorem \ref{the207} is a mixed criterion that also accounts for the initial amplitude. The latter is generally more restrictive but provides additional geometric and quantitative information about the blowup. Both results illustrate how sufficiently large negative slope can lead to wave breaking, even in the presence of weak dissipation.
\end{remark}

\begin{theorem}
\label{the208}
Under the conditions of Theorem \ref{the207}. If $T_{2}^{\ast}$ is finite, then
\begin{equation*}
\lim_{t\rightarrow T_{2}^{\ast}}u_{x}(t,q(t,x_{1}))(T_{2}^{\ast}-t)=-2.
\end{equation*}
\end{theorem}

The proof of Theorem \ref{the208} is similar to that of Theorem \ref{the206}, so
they are omitted for simplicity.

\begin{remark}
Theorems \ref{the206} and \ref{the208} provide two distinct blowup rate results for \eqref{103} with time-dependent dissipation $\lambda(t)$.
While both results yield the same asymptotic rate $-2$, they differ in the blowup scenario: Theorem \ref{the206} is based on a purely pointwise gradient condition, whereas Theorem \ref{the208} exploits the dynamics along characteristics and involves a coupled condition on $u$ and $u_x$. This highlights how different mechanisms-local gradient amplification versus transport-driven singularity formation-can lead to the same universal blowup rate in this dissipative setting.
\end{remark}

\vskip 5mm

{\bf Acknowledgement.}
This work is partially supported by NSFC Grants (nos. 12501210, 12201539, 12561030 and 12225103), Science and Technology Development Plan Project of Jilin Province (no. 20260602002RC), Natural Science Foundation of Gansu Province (no. 26JRRG010) and Doctoral Fund of Hexi University (no. KYQD2025006).


\section*{References}

\begin{thebibliography}{00}
\bibitem{Beals1999} R. Beals, D. Sattinger and J. Szmigielski, Multipeakons and a theorem of Stieltjes, Inverse Probl., 15 (1999) 1--4.

\bibitem{Brandolese2012} L. Brandolese, Breakdown for the Camassa-Holm equation using decay criteria and persistence in weighted spaces, Int. Math. Res. Not. IMRN, 22 (2012) 5161--5181.

\bibitem{Brandolese20141} L. Brandolese, Local-in-space criteria for blowup in shallow water and dispersive rod equations, Commun. Math. Phys., 330 (2014) 401--414.

\bibitem{Brandolese20142} L. Brandolese and M.F. Cortez, On permanent and breaking waves in hyperelastic rods and rings, J. Funct. Anal., 266 (2014) 6954--6987.

\bibitem{Brandolese20143} L. Brandolese and M.F. Cortez, Blowup issues for a class of nonlinear dispersive wave equations, J. Differential Equations, 256 (2014) 3981--3998.

\bibitem{Bressan2007} A. Bressan and A. Constantin, Global conservative solutions of the Camassa-Holm equation, Arch. Ration. Mech. Anal., 183 (2007) 215--239.

\bibitem{Camassa1993} R. Camassa and D. Holm, An integrable shallow water equation with peaked soliton, Phys. Rev. Lett., 71 (1993) 1661--1664.

%\bibitem{Chen20111} R. Chen and Y. Liu, Wave breaking and global existence for a generalized two-component Camassa-Holm system, Int. Math. Res. Not. IMRN, 6 (2011) 1381--1416.

\bibitem{Chen2016} R. Chen, F. Guo, Y. Liu and C. Qu, Analysis on the blow-up of solutions to a class of integrable peakon equations, J. Funct. Anal., 270 (2016) 2343--2374.


%\bibitem{Chen2017} R. Chen, L. Fan, H. Gao and Y. Liu, Breaking waves and solitary waves to the rotation-two-component Camassa-Holm system, SIAM J. Math. Anal., 49 (2017) 3573--3602.

%\bibitem{Chen2018} G. Chen, M. Chen and Y. Liu, Existence and uniqueness of the global conservative weak solutions for the integrable Novikov equation, Indiana Univ. Math. J., 67 (2018) 2393--2433.

%\bibitem{Chen2023} M. Chen, H. Di and Y. Liu, Stability of peaked solitary waves for a class of cubic quasilinear shallow-water equations, Int. Math. Res. Not. IMRN, 7 (2023) 6186--6218.

%\bibitem{Chen2024} M. Chen, L. Fan, X. Wang and R. Xu, Spectral analysis of the periodic b-KP equation under transverse perturbations, Math. Ann., 390(4) (2024) 6315--6354.

\bibitem{Chong2026} G. Chong, Y. Fu and X. Wu, Spectral instability of the periodic peakons in the Novikov equation, J. Differential Equations, 454 (2026) 113938.

\bibitem{Coclite20052} G.M. Coclite, H. Holden and K.H. Karlsen, Global weak solutions to a generalized hyperelastic-rod wave equation, SIAM J. Math. Anal., 37 (2005) 1044--1069.

\bibitem{Constantin19971} A. Constantin, The Cauchy problem for the periodic Camassa-Holm equation, J. Differential Equations, 141 (1997) 218--235.

\bibitem{Constantin1998} A. Constantin and J. Escher, Wave breaking for nonlinear nonlocal shallow water equations, Acta Math., 181 (1998) 229--243.

\bibitem{Constantin19981} A. Constantin and J. Escher, Well-posedness, global existence, and blowup phenomena for a periodic quasi-linear hyperbolic equation, Comm. Pure Appl. Math., 51 (1998) 475--504.

\bibitem{Constantin19982} A. Constantin and J. Escher, On the structure of a family of quasilinear equations arising in a shallow water theory, Math. Ann., 312 (1998) 403--416.

%\bibitem{Constantin19983} A. Constantin and J. Escher, Global existence and blow-up for a shallow water equation, Ann. Sc. Norm. Super. Pisa Cl. Sci. (5), 26 (1998) 303--328.

%\bibitem{Constantin19984} A. Constantin and J. Escher, Global weak solutions for a shallow water equation, Indiana Univ. Math. J., 47 (1998) 1527--1545.

%\bibitem{Constantin1999} A. Constantin and H.P. McKean, A shallow water equation on the circle, Comm. Pure Appl. Math., 52 (1999) 949--982.

\bibitem{Constantin2000} A. Constantin, Existence of permanent and breaking waves for a shallow water equation: a geometric approach, Ann. Inst. Fourier, 50 (2000) 321--362.

\bibitem{Constantin20002} A. Constantin and J. Escher, On the blow-up rate and the blow-up set of breaking waves for a shallow water equation, Math. Z., 233 (2000) 75--91.

\bibitem{Constantin20003} A. Constantin and W.A. Strauss, Stability of peakons, Comm. Pure Appl. Math., 53 (2000) 603--610.

\bibitem{Constantin20004} A. Constantin and L. Molinet, Global weak solutions for a shallow water equation, Comm. Math. Phys., 211 (2000) 45--61.

\bibitem{Constantin2001} A. Constantin, On the scattering problem for the Camassa-Holm equation, Proc. Roy. Soc. London A, 457 (2001) 953--970.

\bibitem{Constantin2002} A. Constantin and L. Molinet, The initial value problem for a generalized Boussinesq equation, Differential Integral Equations, 15 (2002) 1061--1072.

\bibitem{Constantin2009} A. Constantin and D. Lannes, The hydrodynamical relevance of the Camassa-Holm and Degasperis-Procesi equations, Arch. Ration. Mech. Anal., 192 (2009) 165--186.

%\bibitem{DP1999} A. Degasperis and M. Procesi, Asymptotic integrability, in: A. Degasperis, G. Gaeta (Eds.), Asymptotic Integrability, Symmetry and Perturbation Theory, World Scientific, 1999, 23--37.

\bibitem{Fokas1981} A. Fokas and B. Fuchssteiner, Symplectic structures, their B\"acklund transformation and hereditary symmetries, Phys. D, 4 (1981) 47--66.

\bibitem{Freire1} I.L. Freire, N.S. Filho, L.C. Souza and C.E. Toffoli, Invariants and wave breaking analysis of a Camassa-Holm type equation with quadratic and cubic non-linearities, J. Differential Equations, 269 (2020) 56--77.

%\bibitem{Freire2} I.L. Freire, Wave breaking for shallow water models with time decaying solutions, J. Differential Equations, 269 (2020) 3769--3793.

%\bibitem{Freire3} I.L. Freire and C.E. Toffoli, Wave breaking and asymptotic analysis of solutions for a class of weakly dissipative nonlinear wave equations, J. Differential Equations, 358 (2023) 457--483.

\bibitem{Ghidaglia1988} J.M. Ghidaglia, Weakly damped forced Korteweg-de Vries equations behave as a finite dimensional dynamical system in the long time, J. Differential Equations, 74 (1988) 369--390.

%\bibitem{Gui2010} G. Gui and Y. Liu, On the global existence and wave-breaking criteria for the two-component Camassa-Holm system, J. Funct. Anal., 258 (2010) 4251--4278.

\bibitem{Himonas2007} A. Himonas, G. Misiolek, G. Ponce and Y. Zhou, Persistence properties and unique continuation of solutions of the Camassa-Holm equation, Comm. Math. Phys., 271 (2007) 511--512.

%\bibitem{Hone2008} A.N.W. Hone and J. Wang, Integrable peakon equations with cubic nonlinearity, J. Phys. A: Math. Theor., 41(37) (2008) 372002. 10pp.

%\bibitem{Hone2009} A.N.W. Hone, H. Lundmark and J. Szmigielski, Explicit multipeakon solutions of Novikov's cubically nonlinear integrable Camassa-Holm equation, Dyn. Partial Differ. Equ., 6(3) (2009) 253--289.

%\bibitem{Hu2012} Q. Hu, L. Lin and J. Jin, Well-posedness and blowup phenomena for a three-component Camassa-Holm system with peakons, J. Hyperbolic Differential Equations, 9 (2012) 451--467.

\bibitem{Ji2021} S. Ji and Y. Zhou, Wave breaking phenomena and global existence for the weakly dissipative generalized Novikov equation, Proc. Roy. Soc. London A, 477 (2021) 20210532.

\bibitem{Ji2022} S. Ji and Y. Zhou, Wave breaking and global solutions of the weakly dissipative periodic Camassa-Holm type equation, J. Differential Equations, 306 (2022) 439--455.

%\bibitem{Ji20221} S. Ji and Y. Zhou, Global solutions for the modified Camassa-Holm equation, Z. Angew. Math. Mech., 102 (2022) no. e202100567.

%\bibitem{Kang2017} J. Kang, X. Liu, P.J. Olver and C. Qu, Liouville correspondences between integrable hierarchies, SIGMA Symmetry Integrability Geom. Methods Appl., 13 (2017) 1--26.

\bibitem{Kato} T. Kato, Quasi-linear equations of evolution, with applications to partial differential equations, in: Spectral Theory and Differential Equations, Proceedings of the Symposium Dundee, 1974, Dedicated to Konrad Jrgens, in: Lecture Notes in Math., Vol. 448, Springer, Berlin, 1975, pp. 25--70.

\bibitem{Kato1} T. Kato, On the Korteweg-de Vries equation, Manuscripta Math., 28 (1979) 89--99.

\bibitem{Kato1983} T. Kato, On the Cauchy problem for the (generalized) Korteweg-de Vries equation, Studies in Applied Mathematics, Adv. Math. Suppl. Stud., Vol. 8, Academic Press, New York, 1983, pp. 93--128.

\bibitem{Kato1988} T. Kato and G. Ponce, Commutator estimate and the Euler and Navier-Stokes equations, Comm. Pure. Appl. Math., 41 (1988) 891--907.

%\bibitem{Lai2010} S.Y. Lai and Y.H. Wu, Global solutions and blow-up phenomena to a shallow water equation, J. Differential Equations, 249 (2010) 693--706.

\bibitem{Li2000} Y. Li and P. Olver, Well-posedness and blow-up solutions for an integrable nonlinearly dispersive model wave equation, J. Differential Equations, 162 (2000) 27--63.

%\bibitem{Luo2015} W. Luo and Z. Yin, Local well-posedness and blow-up criteria for a two-component Novikov system in the critical Besov space, Nonlinear Anal., 122 (2015) 1--22.

\bibitem{Mustafa20072} O.G. Mustafa, Global conservative solutions of the hyperelastic rod equation, Int. Math. Res. Not. IMRN, 2007 (2007) 1--26.

%\bibitem{Ni2011} L. Ni and Y. Zhou, Well-posedness and persistence properties for the Novikov equation, J. Differential Equations, 250 (2011) 3002--3201.

%\bibitem{Novikov2009} V. Novikov, Generalizations of the Camassa-Holm equation, J. Phys. A: Math. Theor., 42 (2009) 14pp.

\bibitem{Novruzov2022} E. Novruzov and V. Bayrak, Blow-up criteria for a two-component nonlinear dispersive wave system, J. Funct. Anal., 282 (2022) 109454, 19pp.

\bibitem{Ott1970} E. Ott and R.N. Sudan, Damping of solitary waves, Phys. Fluids, 13 (1970) 1432--1434.

\bibitem{Pazy1983} A. Pazy, Semigroup of Linear Operators and Applications to Partial Differential Equations, Springer, New York, 1983.

\bibitem{Popowicz2015} Z. Popowicz, Double extended cubic peakon equation, Phys. Lett. A, 379 (2015) 1240--1245.

%\bibitem{Qiao2003CMP} Z. Qiao, The Camassa-Holm hierarchy, related $N$-dimensional integrable systems and algebro-geometric solution on a symplectic submanifold, Comm. Math. Phys., 239 (2003) 309--341.

%\bibitem{Qiao2004AAM} Z. Qiao, Integrable hierarchy (the DP hierarchy), 3 by 3 constrained systems, and parametric and stationary solutions, Acta Appl. Math., 83 (2004) 199--220.

%\bibitem{Qiao2006JMP} Z. Qiao, A new integrable equation with cuspons and W/M-shape-peaks solitons, J. Math. Phys., 47 (2006) 112701.

%\bibitem{Qiao2024} Z.J. Qiao, E. G. Reyes, Fifth-order equations of Camassa-Holm type and pseudo-peakons, Appl. Numer. Math., 199 (2024) 165--176.

%\bibitem{Qu2010} C. Qu and Y. Fu, On a three-component equation with peakons, Commun. Theor. Phys., 46 (2010) 309--327.

\bibitem{Qu2020} C. Qu and Y. Fu, On the Cauchy problem and peakons of a two-component Novikov system, SCIENCE CHINA Mathematics, 63 (2020) 1965--1996.

%\bibitem{Wahlen2006} E. Wahl\'en, Global existence of weak solutions to the Camassa-Holm equations, Int. Math. Res. Not. IMRN, (2006) 1--12.

\bibitem{Wu2009} S.Y. Wu and Z.Y. Yin, Global existence and blow up phenomena for the weakly dissipative Camassa-Holm equation, J. Differential Equations, 246 (2009) 4309--4321.

\bibitem{Wu2012} X. Wu and Z. Yin, Well-posedness and global existence for the Novikov equation, Ann. Sc. Norm. Super. Pisa Cl. Sci. (5), 11 (2012) 707--727.

\bibitem{Xin2000} Z. Xin and P. Zhang, On the weak solutions to a shallow water equation, Comm. Pure Appl. Math., 53 (2000) 1411--1433.

\bibitem{Xin2002} Z. Xin and P. Zhang, On the uniqueness and large time behavior of the weak solutions to a shallow water equation, Comm. Partial Differential Equations, 27 (2002) 1815--1844.

%\bibitem{Yan2012} W. Yan, Y. Li and Y. Zhang, Global existence and blow-up phenomena for the weakly dissipative Novikov equation, Nonlinear Anal., 75 (2012) 2464--2473.

%\bibitem{Zhang2023} D. Zhang, Y. Zhou, S. Ji and X. Li, On the Cauchy problem for a weakly dissipative coupled Camassa-Holm system, Monatsh. Math., 202 (2023) 857--873.

\bibitem{Zhou2022} Y. Zhou and S. Ji, Well-posedness and wave breaking for a class of integrable equations, Math. Z., 302 (2022) 1527--1550.

\bibitem{Zhou2024} Y. Zhou, S. Ji and Z. Qiao, Globally conservative weak solutions for a class of nonlinear dispersive wave equations beyond wave breaking, J. Differential Equations, 389 (2024) 338--360.

\bibitem{Zhou2025} Y. Zhou, X. Li and S. Ji, Globally conservative weak solutions for a class of two-component nonlinear dispersive wave equations beyond wave breaking, J. Differential Equations, 427 (2025) 538--559.

\bibitem{Zhou2026} Y. Zhou, X. Li, S. Ji and Z. Qiao, A generalized two-component Novikov system and its analytical properties, Physica D, 489 (2026) 135120.
\end{thebibliography}
\end{document}